\documentclass[12pt,reqno]{amsart}
\usepackage{amssymb,amsfonts,latexsym,mathrsfs}

\usepackage[a4paper%
]{geometry}

\usepackage{tikz}
\usetikzlibrary{intersections}

\usepackage{enumitem}
\setenumerate[1]{label=\rm{(\alph*)}}
\setenumerate[2]{label=\rm{(\roman*)}}

\usepackage{graphicx}

\newtheorem{thm}{Theorem}[section]
\newtheorem{cor}[thm]{Corollary}

\newtheorem{lemma}[thm]{Lemma}

\newtheorem{prop}[thm]{Proposition}

\theoremstyle{definition}
\newtheorem{defn}[thm]{Definition}

\newtheorem{example}[thm]{Example}

\newtheorem{rmk}[thm]{Remark}

\newcommand{\N}{\mathbb{N}}
\newcommand{\R}{\mathbb{R}}

\newcommand{\cS}{\mathcal{S}}

\renewcommand{\geq}{\geqslant}
\renewcommand{\leq}{\leqslant}

\theoremstyle{remark}

\DeclareMathOperator{\conv}{conv}
\DeclareMathOperator{\PLC}{PLC}

\DeclareMathOperator{\GL}{GL}
\DeclareMathOperator{\rad}{Rad}

\DeclareMathOperator{\In}{Int}
\DeclareMathOperator{\stab}{Stab}

\renewcommand{\|}{\mid}

\renewcommand{\>}{\rangle}

\newcommand{\x}{\boldsymbol{x}}
\newcommand{\y}{\boldsymbol{y}}

\newcommand{\df}[1]{\textit{\textbf{#1}}}

\renewcommand{\phi}{\varphi}

\newcommand{\Kdavis}{\ensuremath{K_\Sigma}}
\newcommand{\fc}{\ensuremath{K}}
\newcommand{\Klin}{\ensuremath{\widehat{\mathcal{K}}}}
\newcommand{\Zlin}{\ensuremath{{Z}}}

\title{Mapping the Davis complex into the imaginary cone}

\author{Xiang Fu}
\address{Beijing International Center for Mathematical Research, Peking University,  P.R.China}
\email{fuxiang@math.pku.edu.cn}

\author{Lawrence Reeves}
\address{School of Mathematics and Statistics, University~of~Melbourne, Australia}
\email{lreeves@unimelb.edu.au}

\begin{document}
\maketitle

\begin{abstract}
The study of the set of limit roots associated to an infinite Coxeter group was initiated by Hohlweg, Labb\'{e} and Ripoll in \cite{HLR11} and further developed by Dyer, Hohlweg, P\'eaux and Ripoll in 
\cite{DHR13} and \cite{limit3}. 
The Davis complex associated to a finitely generated Coxeter group $W$ is a  piecewise Euclidean CAT(0) space on which $W$ acts properly, cocompactly by isometries. A good reference is the book by Davis \cite{Davis}. 
In this paper we define a natural map from the Davis complex into the normalised imaginary cone of a based root system.

\end{abstract}

\section{Introduction}

Let $(W, S)$ be a Coxeter system. That is $W$ is an abstract group generated by a set $S$ consisting of involutions, subject only to braid relations on elements of $S$.
To study the  Coxeter group $W$, a standard approach is to 
define a representation of $W$ in which the generators act by linear reflections on a vector space $V$.
Such a reflection $\rho$ is typically of the form 
$$
\rho(x)=x-2\frac{(x, a)}{(a, a)}a,
$$     
where $a\in V$ is a representative normal vector to the reflecting hyperplane, and $(\,,\,)$ is some symmetric bilinear form on $V$.
It is a beautiful property of  Coxeter groups  that each element conjugate to an element in $S$ corresponds to one such reflection, and any collection of reflections in $W$ again generate a Coxeter group.  
In the case that $W$ is a finite Coxeter group, it is well known that $(\,,\,)$ is positive definite on the space $V$. 
In the case that $W$ is an infinite Coxeter group, the behaviour of $(\,,\,)$ is less restricted. 
The space $V$ can be partitioned into three regions (some possibly trivial in the case that $W$ is finite): 
\begin{align*}
V^+:&=\{\, v\in V\mid (v, v)>0\,\};\\
V^0:&=\{\,v\in V\mid (v, v)=0 \,\};\\
\noalign{\hbox{and}}\\
V^-:&=\{\,v\in V\mid (v, v)<0\,\}.
\end{align*} 

In either case that $W$ is finite or infinite, the collection of representative normal vectors corresponding to reflections in $W$, known as roots, lie in the region $V^+$.

It is well known that the set of all roots is discrete in $V$, however in the case that $W$ is infinite and $S$ is finite, we can always find suitable hyperplanes $H$ in $V$ such that the projections of all the roots onto $H$ lie in a compact set. 
Hence an infinite sequence of such projected roots (known as \emph{normalised roots}) will possess accumulation points.
It turns out that the set $E$ of all the accumulation points of normalised root lies within $V^0$.
The study of $E$ is of particular significance in understanding the distribution of roots in an arbitrary infinite Coxeter group $W$, 
which is a fundamental  question for general Coxeter group theory yet to be properly investigated.
 
The notion of an \emph{imaginary cone} first appeared in the context of Kac-Moody Lie algebras: it is the cone pointed at the origin and spanned by the positive imaginary roots of the associated Weyl group. This ideal was subsequently generalised into the context of Coxeter groups (see, for example, \cite{Dyer12} and \cite{Fu13}). 
In the Coxeter group setting the imaginary cone $\mathcal{Z}$ lies in the region $V^-$. Furthermore, in \cite{DHR13} it was shown that $\mathcal{Z}$ has an intimate connection with the limit set $E$:
$$
\conv(E) = \overline{Z},
$$
where $Z$ is the projection of $\mathcal{Z}$ onto the hyperplane $H$ and $\overline{Z}$ denotes the topological closure of $Z$.
Therefore, to better understand the set $E$, further investigation into the imaginary cone may well prove to be of value. 
In the literature, not much topological properties of the imaginary cone were known. 
In this paper, by constructing an explicit map from the Davis complex into the normalised imaginary cone $Z$ (which is a $W$-equivariant homeomorphism onto its image), we hope to build a bridge through which known topological properties in the Davis complex may be observed and studied in the imaginary cone.

\section{Based root systems for infinite Coxeter groups}

Let $(W, S)$ be a Coxeter system with parameters $\{\,m_{st}\in \N_{\geq 1}\cup\{\infty\} \mid s, t\in S\,\}$, in the sense of \cite{NB}. That is, $W$ is an abstract group generated by elements of $S$ subject to the relations: 
\begin{itemize}
\item[(1)] $1\notin S$, and $r^2=1$ for all $r\in S$.
\item[(2)] $(st)^{m_{st}}=1$, $\forall s,t\in S$ (in case of $m_{st}=\infty$ then there is no relation between $r$ and $s$); and
$m_{st}=1$ if and only if $s=t$.
\end{itemize}
If $(W, S)$ is a Coxeter system, then the group $W$ is called a Coxeter group. Thus a Coxeter group is an abstract group generated by 
a set of involutions, subject to a set of braid relations specifying the order of products of pairs of generators. These braid relations are governed by the parameters $\{\,m_{st}\mid s, t\in S\,\}$. Often, we study Coxeter groups via their so-called geometric realizations. That is, for a given Coxeter group $W$, we construct a reflection group (a group generated by reflections) isomorphic to 
$W$, and we study $W$ via this isomorphic copy. A key object in constructing such reflection groups is the notion of \emph{root systems}. To construct a root system, one needs the notion of a \emph{root basis} first.

\begin{defn}\textup{(Krammer \cite{K})}
\label{def:datum}
Let $(W, S)$ be a Coxeter system as above.
Suppose that $V$ is a vector space over $\R$ and let $(\,,\,)$ be a bilinear
form on $V$, and let $\Pi=\{\alpha_s\mid s\in S\}$ be a subset of $V$ whose elements are 
in bijective correspondence with $S$. Then $\Pi$ is called a \df{root
basis} if the following conditions are satisfied:
\begin{itemize}
 \item [(C1)] $(\alpha_s, \alpha_s)=1$ for all $s\in S$, and if $s, t$ are distinct elements
of $S$ then either $(\alpha_s, \alpha_t )=-\cos(\pi/m_{st})$ if
$m_{st}$ if finite, or else $(\alpha_s, \alpha_t) \leq -1$ if
$m_{st}=\infty$);
 \item [(C2)] $0\notin \PLC(\Pi)$, where for any set $A$, $\PLC(A)$ denotes the
set
 $$\{\,\sum\limits_{a\in A} \lambda_a a\mid \text{$\lambda_a \geq 0$ for all
$a\in A$ and $\lambda_{a'}>0$ for some $a'\in A$}\,\}.$$ 
\end{itemize}
\end{defn}

If $\Pi$ is a root basis, then we call the triple $\mathscr{C}=(\,V, \, \Pi,
\,(\,,\,)\,)$ a \df{Coxeter datum}. Throughout this section we fix an  arbitrarily chosen
Coxeter datum $\mathscr{C}$. Observe that (C1) implies that for each $a\in \Pi$,
$a\notin \PLC(\Pi\setminus\{a\})$. Furthermore, (C1) together with (C2) yield
that whenever $a, b\in \Pi$ are distinct then $\{a, b\}$ is linearly
independent. 

For each non-isotropic $x\in V$ (that is, $(x, x)\neq 0$) define $\rho_x \in \GL(V)$, the \emph{reflection corresponding to $x$} by the rule: 
$$\rho_x v=v-2\frac{(x, v)}{(x, x)}x,$$ for all $v\in V$. Note that $\rho_x$ is an involution and $\rho_x
x=-x$. The following proposition summarizes a few useful results:
 
\begin{prop}\textup{\cite[Lecture 1]{RB96}}
 \label{pp:anu1}
\rm{(i)}\quad Let $r, s\in S$ be distinct with $m_{rs}\neq
\infty$. Set $\theta =\pi/m_{rs}$. Then for each integer $i$,
$$(\rho_{\alpha_r} \rho_{\alpha_s})^i \alpha_r=\frac{\sin(2i+1)\theta}{\sin \theta}\alpha_r+\frac{\sin
2i\theta}{\sin\theta}\alpha_s, $$
and in particular, $\rho_{\alpha_r} \rho_{\alpha_s}$ has order $m_{rs}$.

\noindent\rm{(ii)}\quad Let $r, s\in S$ be distinct with
$m_{rs}=\infty$. Set $\theta =\cosh^{-1}(-(\alpha_r, \alpha_s))$. Then for each integer $i$,
\begin{equation*}
(\rho_{\alpha_r} \rho_{\alpha_s})^i \alpha_r=
\begin{cases}
 \frac{\sinh(2i+1)\theta}{\sinh \theta}\alpha_r+\frac{\sinh
2i\theta}{\sinh\theta}\alpha_s, \text{ if $(\alpha_r, \alpha_s) \neq -1$}\\
(2i+1)\alpha_r+2i \alpha_s, \text{ }\text{ }\text{ }\text{ }\text{ }\text{ }\text{ }\text{ if $(\alpha_r, \alpha_s)=-1$}, 
\end{cases}
\end{equation*}
and in particular, $\rho_{\alpha_r} \rho_{\alpha_s}$ has infinite order.
\qed
\end{prop}

Let $G_{\mathscr{C}}=\langle \{\,\rho_a\mid a\in \Pi\,\} \rangle\subseteq \GL(V)$,
the subgroup of $\GL(V)$ generated by the involutions
in the set $\{\,\rho_a\mid a\in \Pi\,\}$.
Then Proposition~\ref{pp:anu1} yields that there is a group homomorphism
$\phi_{\mathscr{C}}\colon W\to G_{\mathscr{C}}$ satisfying
$\phi_{\mathscr{C}}(r_a)=\rho_a$ for all $a\in \Pi$. This homomorphism together
with the $G_{\mathscr{C}}$-action on $V$ give rise to a $W$-action on $V$: for
each $w\in W$ and $x\in V$, define $wx\in V$ by $wx=\phi_{\mathscr{C}}(w)x$. It
can be easily checked that this $W$-action preserves $(\,,\,)$.
Denote the length function of $W$ with respect to $S$ by $\ell$. Then we have:

\begin{prop}\textup{\cite[Lecture 1]{RB96}}
 \label{pp:anu2}
Let $G_{\mathscr{C}}, W, S$ be as the above and let $w\in W$ and $a\in \Pi$. If
$\ell(wr_a)\geq \ell(w)$ then $wa\in \PLC(\Pi)$.
\qed
\end{prop}
Immediately from the above proposition we have:

\begin{thm} \textup{(\cite[Lecture 1]{RB96})}
 \label{co:anu2}
The above map
$$\phi_{\mathscr{C}}\colon W\to G_{\mathscr{C}}$$ 
is an isomorphism.
\qed
\end{thm}

In particular, the above theorem yields that $(G_{\mathscr{C}}, \{\,\rho_a\mid
a\in \Pi\, \})$ is a Coxeter system isomorphic to $(W, S)$. We call $(W, S)$ the
\emph{abstract Coxeter system} associated to the Coxeter datum $\mathscr{C}$ and
we call $W$ a Coxeter group of \emph{rank} $\#S$, where $\#$ denotes
cardinality. 

\begin{defn}
\label{def:root system}
The \emph{root system} of $W$ in $V$ is the set 
$$\Phi=\{\,wa \mid \text{$w\in W$ and $a\in \Pi$}\,\}.$$
The set $\Phi^+=\Phi\cap \PLC(\Pi)$ is called the set of \emph{positive roots}
and the set $\Phi^-=-\Phi^+$ is called  the set of \emph{negative roots}.
\end{defn}
From Proposition \ref{pp:anu2} and Theorem \ref{co:anu2} we may readily deduce
that:
\begin{prop}\textup{(\cite[Lecture 3]{RB96})}
\label{pp:anu3}
\rm{(i)}\quad Let $w\in W$ and $a\in \Pi$. Then 
 \begin{equation*}
\ell(wr_a) =
\begin{cases}
\ell(w)-1  \text{   if } wa\in \Phi^-\\
\ell(w)+1  \text{   if } wa\in \Phi^+.
\end{cases}
\end{equation*}

\noindent\rm{(ii)}\quad $\Phi=\Phi^+\biguplus\Phi^-$, where $\biguplus$ denotes
disjoint union.

\noindent\rm{(iii)}\quad $W$ is finite if and only if $\Phi$ is finite.
\qed
\end{prop}

Let $T=\bigcup_{w\in W} w S w^{-1}$, and we call it the set of \emph{reflections}
in $W$. Let $x$ be in $\Phi$. It follows that $x=wa$ for
some $w\in W$ and $a\in \Pi$. Direct calculations yield that $\rho_x
=(\phi_{\mathscr{C}} (w)) \rho_a (\phi_{\mathscr{C}} (w))^{-1}\in G_{\mathscr{C}}$.
Now let $r_x\in W$ such that $\phi_{\mathscr{C}}(r_x)=\rho_x$. Then $r_x = w r_a
w^{-1}\in T$ and we call it the \emph{abstract reflection} corresponding to $x$. It is readily
checked that $r_x =r_{-x}$ for all $x\in \Phi$ and $T=\{\,r_x \mid x\in
\Phi\,\}$. For each $t\in T$ we let $\alpha_t$ be the unique positive root with
the property that $r_{\alpha_t}=t$. It is also easily checked that there is a
bijection $T \leftrightarrow \Phi^+ $ given by $t \to \alpha_t $ ($t\in T$), and 
$x\to \phi_{\mathscr{C}}^{-1}(\rho_x)$ ($x\in \Phi^+ $). We call this bijection
the \emph{canonical bijection} between $T$ and $\Phi^+$.


We close this preliminary section with the following well-known results on finite Coxeter groups:
\begin{lemma}
\label{lem:finite}
Let $\mathscr{C}=(V, \Pi, (\,,\,))$ be a Coxeter datum. Let $\Pi'\subseteq \Pi$, and let
$V_{\Pi'}$ be the subspace of $V$ spanned by $\Pi'$, and furthermore let $W_{\Pi'}:=\langle \{\,r_a\mid a\in \Pi'\,\}\rangle$.
Then $W_{\Pi'}$ is finite if and only if the restriction of the bilinear form $(\,,\,)$ to the subspace $V_{\Pi'}$ is positive definite.\qed 
\end{lemma}

\section{Transverse hyperplanes}
\begin{defn}\textup{(\cite{HLR11}, \cite{DHR13})}
\label{def: trans}
Given a Coxeter datum $\mathscr{C}=(V,\Pi,B)$, an affine hyperplane $V_1$ of codimension $1$ in $V$ is 
called \df{transverse} (to $\Phi^+$) if for each simple root $a\in \Pi$ the ray $\R_{>0}a$  intersects $V_1$ in
exactly one point, and this unique intersection is denoted by $\widehat{a}_{V_1}$. Given a hyperplane $V_1$ transverse to 
$\Phi^+$, let $V_0$ be the hyperplane that is parallel to $V_1$ and contains the origin.
\end{defn}

\begin{rmk}
For a Coxeter datum $\mathscr{C}$, it follows from Proposition~\ref{pp:anu3}~(ii) that it is 
always possible to find a hyperplane containing the origin that separates $\Phi^+$ and $\Phi^-$.
By suitably translating this hyperplane it is always possible to find a hyperplane transverse to  $\Phi^+$.
\end{rmk}

\begin{rmk}
\label{P}
Let $V_1$ be a transverse hyperplane and let $V_0$ be as in the preceding definition. 
Let $V_0^{+}$ be the open hall space induced by $V_0$ that contains $V_1$. Observe that 
$V_0^{+}$ contains  $\PLC(\Pi)$. Since 
$\Phi_{\mathscr{C}}^+\subset \PLC(\Pi)\subset V_0^{+}$, and $V_1$ is parallel to the boundary of 
$V_0^{+}$, it follows that $|\,V_1\cap \R_{>0} \beta\,| =1$ for each $\beta\in \Phi^+_{\mathscr{C}}$
(where $|A|$ of a set $A$ denotes the cardinality of $A$). 
\end{rmk}

The above remark leads to an alternative definition of transverse hyperplanes:

\begin{lemma}\textup{\cite{HLR11}}
\label{lem: trans}
Given a Coxeter datum $\mathscr{C}=(V,\Pi,B)$, an affine hyperplane $V_1$
is transverse if and only if $|\,V_1\cap \R_{>0} \beta\,| =1$ for each $\beta\in \Phi^+_{\mathscr{C}}$.
\qed
\end{lemma}

\begin{defn}
\label{df: norm}
Let $\mathscr{C}=(V,\Pi,B)$ be a Coxeter datum, 
let  $V_1$ be a transverse hyperplane  in $V$, and let $V_0$ be obtained from $V_1$ as in Definition~\ref{def: trans}. 
\begin{enumerate}

\item For each $v\in V\setminus V_0$,  the unique intersection point of $\R v$ 
and the transverse hyperplane $V_1$  is denoted $\widehat{v}$. The 
\df{normalization map} is $\pi_{V_1}\colon V\setminus V_0\to V_1$, $\pi_{V_1}(v)=\widehat{v}$.
Set $\widehat{\Phi}_{\mathscr{C}{V_1}}=\pi_{V_1}(\Phi_{\mathscr{C}})$, and the elements 
$\widehat{x}\in\widehat{\Phi}_{\mathscr{C}{V_1}}$ are called \df{normalised roots}.

\item For $\widehat{x}\in \widehat{\Phi}_{\mathscr{C}V_1}$, set $x^+$ to be the 
unique element in $\Phi_{\mathscr{C}}^+$ with $\pi_{V_1}(x^+)=\widehat{x}_{V_1}$. 

\item Let $\varphi_{V_1}\colon V\to \R$ be the unique linear map satisfying the requirement that
$\varphi_{V_1}(v)=0$ for all $v\in V_0$, and $\varphi_{V_1}(v)=1$ for all $v\in V_1$.
\end{enumerate}
\end{defn}

Observe that $\pi_{V_1}(-x)= \pi_{V_1}(x)$ for all $x\in V$, and  $\widehat{y} =\frac{y}{\varphi_{V_1}(y)}$ for all $y\in V\setminus V_0$. 

For any $X\subseteq V\setminus V_0$, we set $\widehat{X}:=\{\,\widehat{x}\mid x\in X\,\}$. 

\begin{rmk}
\label{D}
Remark \ref{P} above implies that $\PLC(\Pi)\cap V_0=\emptyset$, and hence any subset of $\PLC(\Pi)$ can be normalised. 
\end{rmk}

It has been observed in \cite{HLR11} that the action of $W$ on $V$ induces, via the normalization map $\pi_{V_1}$,  a natural action
on the following region $D$ of $V_1$:
\begin{defn}
\label{dot}
\begin{align*}
D:&=(\bigcap_{w\in W} w(V\setminus V_0)) \cap V_1\\
  &= V_1\setminus \bigcup_{w\in W}w V_0.
\end{align*}   
As in \cite{HLR11} we define the $\cdot$-action on $D$ as follows: for $x\in D$, and $w\in W$, set
$$
w\cdot x=\widehat{wx}.
$$
\end{defn}
Hence if $x\in D$ then $wx\in V\setminus V_0$ for all $w\in W$. It is readily checked that the $\cdot$-action is a well defined
action of $W$, and that any $w\in W$ acts continuously on $D$ with respect to this $\cdot$-action. The set $D$ is the maximal subset of $V_1$ on which $W$ acts naturally.  

\section{Davis complex}

In this section we briefly recall the construction of the Davis complex. The reader should consult \cite{Davis} (where it is referred to as $\Sigma$) for further background and details.
%
%
%
%
%
A subset $T\subseteq S$ is called \df{spherical} if $W_T:=\langle r\mid r\in T \rangle $ is a finite.
We also call $\Pi'\subseteq \Pi$ \df{spherical} if $W_{\Pi'}$ is finite. 
Let $\cS$ denote the poset of all spherical subsets of $S$ ordered by inclusion.
The \df{fundamental chamber}, denoted \Kdavis, is the geometric realisation of the the poset $\cS$.
That is, there is a vertex in \Kdavis\ for each spherical subset of $S$, and vertices $v_1,\dots, v_k$ span a $(k-1)$-simplex in \Kdavis\ if the corresponding collection of spherical subgroups is totally ordered by inclusion. 
For  $s\in S$ let $\fc_s$ denote the union of those simplices of \fc\ that have $\{s\}$ as the minimal element.
Define an equivalence relation on the set $W\times \fc$ by setting
$$
(w,k)\sim(w',k') \quad \text{if} \quad k=k' \quad\text{and}\quad w^{-1}w'\in W_k, 
$$
where $W_k$ is the stabiliser of $k$. 

The \df{Davis complex} is the simplicial complex given as the quotient of $W\times \fc$ by the above equivalence relation.

%
%

\begin{example}
For the Coxeter group $W=\<s,t\|s=t^2=(st)^3=1\>$ we have $K$ and $\Sigma$ as shown below.
$$
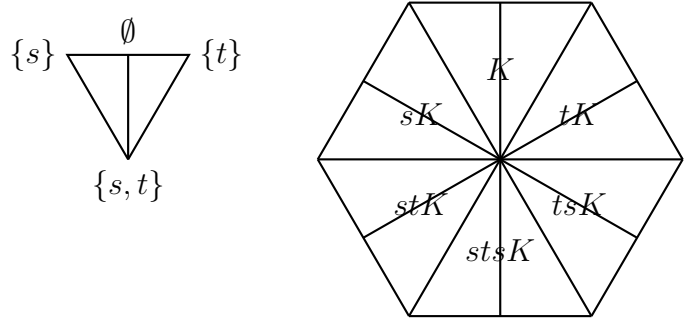
\begin{tikzpicture}[scale=0.4,baseline=-20pt]

\def\rad{4}


\coordinate [label=below:{$\{s,t \}$}] (O) at (0,0);
\coordinate [label=left:{$\{s \}$}] (S) at (120:\rad);
\coordinate  [label=right:{$\{t \}$}] (T) at (60:\rad);
\path [name path=vert] (0:0)--(90:\rad);
\path [name path=top] (120:\rad)--(60:\rad);


\draw [name intersections={of=vert and top, by=E}]  
         [thick] (O) -- (E);
         
\draw (E) node [above] {$\emptyset$};


\draw [thick] (O) -- (S) -- (T) -- (O);

\end{tikzpicture}
\qquad
\begin{tikzpicture}[scale=0.8,baseline=-20pt,thick]

\def\rad{3}

\path [name path=vert] (0:0)--(90:\rad);

\path [name path=top] (120:\rad)--(60:\rad);

\path [name intersections={of=vert and top, by=F}]  
         [thick] (0,0) -- (F);

\foreach \x / \y in {0/{},60/s,120/st,180/sts,240/ts,300/t}
{
\draw [rotate=\x,name path=side] (120:\rad)--(60:\rad);
\draw [rotate=\x] (0,0)--(60:\rad);

\path [rotate=\x, name path=spoke] (0,0)--(0,\rad);
\draw [name intersections={of=side and spoke}]  
         [thick] (0,0) -- (intersection-1);
\path [rotate=\x] (0,0)--(0,0.5*\rad) node {$\y K$};
}

\end{tikzpicture}
$$

\end{example}

\section{Imaginary cone}


For the rest of this paper, let $(W, S)$ be a Coxeter system with $|S|<\infty$, and let $(V, \Pi, (\,,\,))$ be a corresponding Coxeter datum, 
and let $\Phi$ be the associated root system. 
%

For each $x\in V$, define $H_x=\{\,v\in V\mid (x, v)=0 \,\}$, the \df{orthogonal hyperplane} corresponding to  $x$.
%
%
%
Following \cite{DHR13}, we set 
\begin{align*}
\mathcal{K}&=\{\, v\in \PLC(\Pi)\mid (v, a)\leq 0, \text{ for all $a\in \Pi$} \,\}.
\end{align*}
In particular, $\mathcal{K}\cap V_0=\emptyset$.

The next result was observed in \cite[Lemma 2.4]{DHR13}:
\begin{lemma}\label{nonempty}
Suppose that $W$ is an irreducible and infinite Coxeter group which is not affine. Then $\In(\mathcal{K})\neq \emptyset$ (where $\In$ denotes the topological interior).\qed
\end{lemma}
If $x\in \In(\mathcal{K})$, then it is clear form the definition that $(x, a)<0$ for all $a\in\Pi$, and in particular, we have
\begin{equation}
\label{eq:K}
(x, x)<0.
\end{equation}

\begin{defn}
The \df{imaginary cone} $\mathcal{Z}$ of $W$ in $V$ is defined by 
$$
\mathcal{Z}=\bigcup_{w\in W} w \mathcal{K}. 
$$
\end{defn}
It is clear from the above definition that $\mathcal{Z}$ is $W$-invariant. 

The following technical result was observed in \cite{Fu13}, and since it is used repeatedly in this paper, we include a proof here.
\begin{lemma}\label{4.5}
Suppose that $v\in V$ has the property that $(v, a)\leq 0$ for all $a\in \Pi$. Then $wv-v\in \PLC(\Pi)\cup \{0\}$ for all $w\in W$.
\end{lemma}
\begin{proof}
Use an induction on $\ell(w)$. If $\ell(w)=0$, then there is nothing to prove.
If $w\neq 1$, we may write $w=w' r_a$ where $a\in \Pi$ and $w'\in W$ with $\ell(w)=\ell(w')+1$.
Then Proposition~\ref{pp:anu3} gives us that $w' a\in \Phi^+$, and hence 
$$w v-v=(w' r_a)v-v =w'(v-2(v, a) a)-v=(w'v-v)-2(a, v)w'a.$$
By the inductive hypothesis, $w' v-v\in \PLC(\Pi)\cup \{0\}$. 
Since $(v, a)\leq 0$, it follows from $w'a\in \Phi^+$ that 
$wv-v\in \PLC(\Pi)\cup\{0\}$.
\end{proof}

\begin{rmk}
\label{K}
A direct consequence of Lemma \ref{4.5} and Remark~\ref{D} is that for all $w\in W$, 
$$w\mathcal{K}\subseteq \PLC(\Pi),$$ 
and hence $\mathcal{\Klin}$ is a subset of $D$, and the $\cdot$-action is defined on $\mathcal{K}$. This then implies that
the imaginary cone $\mathcal{Z}$ is a subset of $\PLC(\Pi)$. Therefore, we can normalise the imaginary cone $\mathcal{Z}$, 
and following the notation in \cite{DHR13} we set 
$$Z:=\widehat{\mathcal{Z}}.$$ 
Since $\mathcal{Z}$ is $W$-invariant, it follows that $Z\subseteq D\subseteq V_1$, and the $\cdot$-action is defined on
$Z$.  
\end{rmk}




\begin{lemma}\label{intersection}
Suppose that $W$ is a non-affine infinite Coxeter group, and suppose that 
$s_1, \ldots, s_n\in S$ such that $W':=\langle \{s_1, \ldots, s_n\}\rangle$ is a finite
standard parabolic subgroup of $W$. Then 
$(H_{s_1}\cap \cdots \cap H_{s_n})\cap \Klin\neq \emptyset$.

In particular, if $v_0\in \In(\mathcal{K})$ then 
$$
\frac{1}{|W'|}\sum_{w\in W'} w v_0\in \mathcal{K} .
$$
\end{lemma}
\begin{proof}
Lemma \ref{nonempty} above yields that $\In \mathcal{K}\neq \emptyset$. Pick $v_0\in \In \mathcal{K}$, and define
$c=\sum_{w\in W'} w v_0$. It then follows that for each $s_i$ ($i=1, \ldots, n$), we have
$s_i c=c$. Hence $(c, \alpha_{s_i})=0$, for all $i=1, \ldots, n$, whence $c\in H_{s_i}\cap\cdots \cap H_{s_n}$.
Observe that Lemma \ref{4.5} then yields that
$c-v_0 =\sum_{i}^{n} \lambda_i \alpha_{s_i}+(|W'|-1) v_0$,
where $\lambda_i \geq 0$ for all $i=1, \ldots, n$. Consequently, for all $t\in S\setminus\{s_1, \ldots, s_n\}$, 
\begin{align*}
(c, \alpha_t)&= ((|W'|-1)v_0+\sum_{i=1}^{n}\lambda_i \alpha_{s_i}, \alpha_t)\\
             &=(|W'|-1)(v_0, \alpha_t)+\sum_{i=1}^{n}\lambda_i(\alpha_{s_i}, \alpha_t)\\
						 &\leq (|W'|-1) (v_0, \alpha_t)\\
						 &<0,
\end{align*}
since $t\neq s_1, \ldots, s_n$ and $v_0\in \In \mathcal{K}$. 
Thus $(c, \alpha_s)\leq 0$ for all $s\in S$, and hence $c\in \mathcal{K}$.
It is readily checked that $\mathcal{K}$ is closed  under multiplication with a positive constant, 
consequently, 
$$c'=\frac{1}{|W'|}c=\frac{1}{|W'|}\sum_{w\in W'} w v_0\in \mathcal{K}.$$ 
Finally, $\widehat{c}\in (H_{s_1}\cap \cdots \cap H_{s_n})\cap \widehat{\mathcal{K}}$.
\end{proof}

\begin{lemma}\label{aff}
Suppose that $s_1, \ldots, s_n\in S$ such that $\langle s_1, \ldots, s_n\rangle$ is an (irreducible) affine 
reflection subgroup. Then $(H_{s_1}\cap \cdots \cap H_{s_n})\cap \widehat{\mathcal{K}}\neq \emptyset$
\end{lemma}
\begin{proof}
If $\langle s_1, \ldots, s_n\rangle$ is an (irreducible) affine 
reflection subgroup then it is well known that $(\,,\,)$ restricted 
to the subspace spanned by $\alpha_{s_1}, \ldots, \alpha_{s_n}$ has 
a one dimensional radical spanned by an element $\alpha$ of the form
$\alpha:=\sum_{1=1}^{n} \lambda_i \alpha_{s_i}$ with $\lambda_i >0$ for all
$i=1, \cdots, n$. Then clearly 
$(\alpha, \alpha_{s_i})=0$ for all $i=1, \cdots, n$  and
$(\alpha, \alpha_t)\leq 0$ for all $t\in S\setminus\{s_1, \ldots, s_n\}$.
Consequently, $\alpha\in (H_{s_1}\cap \cdots \cap H_{s_n})\cap \widehat{\mathcal{K}}$.
\end{proof}

\begin{prop}\label{aff&fin}
Suppose that $s_1,\ldots, s_n\in S$. Then $(H_{s_1}\cap \cdots \cap H_{s_n})\cap \widehat{\mathcal{K}}\neq \emptyset$
if and only if either $\langle s_1, \ldots, s_n\rangle$ is finite or affine.
\end{prop}
\begin{proof}
Lemma \ref{intersection} and Lemma \ref{aff} cover the if part.

Conversely, suppose that $(H_{s_1}\cap \cdots \cap H_{s_n})\cap \widehat{\mathcal{K}}\neq \emptyset$, and
set $V'$ to be the subspace spanned by $\alpha_{s_1}, \ldots, \alpha_{s_n}$, and furthermore, let $V'^{\perp}$ be the
radical of $(\,,\,)$ restricted to $V'$. First, suppose that $\langle s_1, \ldots, s_n\rangle$ is infinite. Then by Proposition 4.8 of [the first limit paper], the fact that $(H_{s_1}\cap \cdots \cap H_{s_n})\cap \widehat{\mathcal{K}}\neq \emptyset$ implies that $\langle s_1, \ldots, s_n\rangle$
is affine. On the other hand, if $\langle s_1, \ldots, s_n\rangle$ is non-affine. Then by Proposition 4.8 of \cite{HLR11} 
the fact that $(H_{s_1}\cap \cdots \cap H_{s_n})\cap \widehat{\mathcal{K}}\neq \emptyset$ implies that $\langle s_1, \cdots, s_n\rangle$ is not  infinite.
\end{proof}

\begin{lemma}\label{stab}
Suppose that $v\in V$ satisfies the condition that $(v, a)\leq 0$ for all $a\in \Pi$, and
let $S':=\{\, s\in S\mid (v, \alpha_s)=0 \,\}$.
Then $W_v=\langle S'\rangle$, where $W_v$ is the stabiliser of $v$.
\end{lemma}
\begin{proof}
It is clear that $\langle S' \rangle \subseteq W_v$, since $(v, \alpha_s)=0$ for 
all $s\in S'$.

On the other hand, suppose that $w\in W_v$ with $w\neq 1$. We may write
$w= w' r_a$ where $a\in \Pi$ and $\ell(w)=\ell(w')+1$. Note that then $w'a\in \Phi^+$, and
$$v=w v = w' r_a v= w' v-2 (a, v) w' a.$$ That is, 
$w'v-v=2(a, v)w'a$
It follows from Lemma~\ref{4.5} that $(a, v)\geq 0$, forcing
$(a, v)=0$, and $r_a\in S'$. 
It then follows by an induction that $w\in \langle S' \rangle$.

\end{proof}

\begin{lemma}\label{K&Q}
Suppose that $x\in \mathcal{K}\cap Q$, and $x\neq 0$. Then there exists a subset $M$ of $\Pi$ such that
the restriction of the bilinear form $(\,,\,)$ to the subspace spanned by $M$ has 
a nonzero radical $\rad(M)$, and $x\in \rad(M)$.
\end{lemma}
\begin{proof}
Let $M:=\{\, a\in \Pi\mid (x, a)=0\,\}$, and for each $a\in \Pi$, let $\lambda_a$ be the 
coefficient of $a$ in $x$. Thus $x=\sum_{a\in \Pi} \lambda_a a$. Note that each of $\lambda_a$ is 
non-negative, as $x\in \mathcal{K}$.

Now $0=(x, x)=\sum_{a\in \Pi}\lambda_a (x, a)$. Those terms corresponding to $a\in J$ are all zero.
For the others, $(x, a)$ is strictly negative. Since $(x, x)=0$, it follows that $\lambda_a=0$ for 
all those $a\in \Pi\setminus M$, and so $x$ is in the span of $M$. Since $(x, a)=0$ for all $a\in M$, we 
see that $x\in \rad(M)$.
\end{proof}

\begin{prop}\label{positive}
Let $u_0\in \In(\mathcal{K})$ be arbitrary, and let $\Pi'\subseteq \Pi$ be such that the standard parabolic subgroup $W_{\Pi'}:=\langle r_a\mid a\in \Pi'\rangle$ is
finite. Then the set 
$$W_{\Pi'} u_0:=\{\,w u_0\mid w\in W_{\Pi}\,\}$$
is positively independent, that is, $0\notin \PLC(W_{\Pi'} u_0)$. 
\end{prop}
\begin{proof}
First note that since $u_0\in \In(\mathcal{K})$, it follows that $(u_0, a)<0$ for all $a\in \Pi$. Then Lemma \ref{stab} above yields that the stabiliser of
$u_0$ is trivial, and consequently all elements in the set $W_{\Pi'} u_0$ are distinct from each other.  
Now suppose for a contradiction that $W_{\Pi'} u_0$ is not positively independent. Then without loss of generality, we may assume that the following holds:
$$u_0 =\sum_{w\in W_{\Pi'}\setminus\{1\}}\lambda_w w u_0, $$
where all the $\lambda_w \geq 0$. Pick an arbitrary $w_0\in W_{\Pi'}\setminus\{1\}$, we have 
\begin{equation}
\label{diff}
u_0-\lambda_{w_0} w_0 u_0 =\sum_{w\in W_{\Pi'}\setminus\{1, w_0\}} \lambda_w w u_0.
\end{equation}
Observe that by Lemma \ref{4.5} the right hand side of equation (\ref{diff}) is in $\PLC(\Pi)$. 
If $\lambda_{w_0}\geq 1$ then Lemma~\ref{4.5} says that the left hand side of equation~\ref{diff}
is a non-positive linear combination of $\Pi$, which is clearly absurd. 
If $\lambda_{w_0}<1$, then the fact that the right hand side of equation (\ref{diff}) is in $\PLC(\Pi)$
and Lemma~\ref{4.5} together force $u_0$ and $w_0 u_0$ to be scalar multiples of each other. But this
implies that $u_0$ lies in the vector subspace spanned by $\Pi'$, and this is clearly impossible because 
by Lemma~\ref{lem:finite} the restriction of the bilinear form $(\,,\,)$ on this subspace is positive definite and yet 
by (\ref{eq:K}) we have that $(u_0, u_0)<0$.
\end{proof}

Noting that $\pi_{V_1}(u)$ is strictly positive for each $u\in \In(\mathcal{K})$, and hence 
we may easily obtain the following:

\begin{cor}
Let $u_0\in \In(\mathcal{\Klin})$ be arbitrary, and let $\Pi'\subseteq \Pi$ be such that the standard parabolic subgroup $W_{\Pi'}$ is
finite. Then the set 
$$W_{\Pi'} \cdot u_0:=\{\,w \cdot u_0\mid w\in W_{\Pi}\,\}$$
is positively independent.\qed
\end{cor}

\section{Defining the map}
For this section, we assume that $\Pi$ is a basis for $V$, and we take the transverse hyperplane $V_1$ to be 
$$
V_1=\{\,v\in V\mid \sum_{a\in \Pi} v_a =1\,\}, 
$$
where the $v_a$'s are the coordinates of $v$ in the basis $\Pi$.
Consequently, 
$$
V_0=\{\,v\in V\mid \sum_{a\in \Pi} v_a =0\,\}, 
$$
and naturally, the normalization map $\pi_{V_1}$ is explicitly given by the rule:  for any $v\in V\setminus V_0$, 
$$
\pi_{V_1} (v)=\sum_{a\in \Pi} v_a.
$$

We begin by fixing a base point $v_0\in\In(\Klin)$. We then define a map $F:\Kdavis\to\Klin$ and  then extend to the whole Davis complex $\Sigma$.
Given a spherical subset $T\subseteq S$, denote by $u_T\in\Kdavis$ the corresponding vertex in the fundamental chamber of the Davis complex. The map will send $u_\emptyset$ to the base point $v_0$. 
As the image of $u_T$ we will choose the  element  $v_T\in\Zlin$
given by $$v_T=\frac{1}{|W_T|}\sum_{g\in W_T} g\cdot v_0$$ 
%
Note that the same argument as that used in the proof of Lemma~\ref{intersection} enables us to deduce that $v_T\in\Klin$. Note also that $t\cdot v_T=v_T$ for all $t\in T$.
In particular if $t\in T$, then $v_T\in H_t$.  
	For $T\neq\emptyset$ we have  $v_T\in {\Klin}\cap(\cap_{t\in T}H_t)$ and therefore $v_T$ lies in the frontier of $\Klin$.

\begin{lemma}	\label{lem:chain}
Let $T:=\{\,t_1, t_2, \ldots, t_k\,\}$ be a spherical subset of $S$, and for each $i\in\{\,1, 2, \ldots, k\,\}$, define
$T_i:=\{\,t_1, t_2, \ldots, t_i\,\}$, that is, $T_1\subsetneq \dots\subsetneq T_k$ is a chain of spherical subsets of $S$.
Then the corresponding vertices $v_{T_1},\dots,v_{T_k}$ span a $(k-1)$-simplex in $H_s\cap{\Klin}$.
\end{lemma}
\begin{proof}
It is enough to prove that $v_{T_k}$ can not be expressed in the form $\sum_{i=1}^{k-1}\lambda_i v_{T_{i}}$ where 
$\sum_{i=1}^{k-1}\lambda_i =1$. Suppose for a contradiction that for some $\lambda_1, \ldots,\lambda_k$ satisfying $\sum_{i=1}^{k-1}\lambda_i =1$ and we have
\begin{equation}
\label{eq:v_k}
\frac{1}{|W_{T_k}|}\sum_{w\in W_{T_k}}w\cdot v_0=\sum_{i=1}^{k-1}\lambda_i(\frac{1}{|W_{T_i}|} \sum_{w\in W_{T_i}} w\cdot v_0).
\end{equation}
Noting that the left hand side of (\ref{eq:v_k}) can be expressed as
\begin{equation}
\label{eq:t}
\frac{1}{|W_{T_k}|}(t_k\cdot v_0 +\sum_{w\in W_{T_k}\setminus\{t_k\}}w\cdot v_0). 
\end{equation}
Given $v_0\in \In(\mathcal{K})$, we must have $t_k\cdot v_0 =\lambda v_0 + \mu \alpha_t$, where $\lambda$ and $\mu$
are strictly positive and $\alpha_t$ is the simple root corresponding to $t$.
Furthermore, Lemma~\ref{4.5} yields that all the terms $w\cdot v_0$ appearing in (\ref{eq:t}) are in $\PLC(\Pi)$.
Since here we assumed that $\Pi$ is linearly independent, and $t_k$ is not in any of the $W_{T_i}$ where $i\in \{1, 2, \ldots, k-1\}$, then (\ref{eq:v_k}) forces $v_0$ to be in the vector subspace
spanned by $\alpha_{t_1}, \ldots, \alpha_{t_k}$. But again as in Proposition~\ref{positive}, this is clearly impossible as the bilinear form $(\,,\,)$ is positive 
definite on this subspace (by Lemma~\ref{lem:finite}) and yet $(v_0, v_0)<0$.
\end{proof}

Each simplex of \Kdavis\ is given by such a chain of spherical subsets. Having chosen the map on vertices of \Kdavis\ to be $F(u_t)=v_T$ the above lemma allows us to extend to the whole simplex in \Kdavis\ spanned by the $u_{T_i}$.

\begin{lemma}
Let $k\in\Kdavis$ and $s\in S$. Then $k\in K_s$ if and only if $F(k)\in H_s$.
\end{lemma}
\begin{proof}
Let $k\in \Kdavis$ and $s\in S$.
It follows from Lemma~\ref{lem:chain} that, for all $s\in S$ and spherical $T\subset S$, we have $s\cdot v_T=v_T$ if and only if $s\in T$.
If the simplex of $\Kdavis$ containing $k$ is given by $T_1\subsetneq \cdots\subsetneq T_n$, then we have
$$
k\in K_s \iff s\in T_1 \iff s\cdot v_{T_i}=v_{T_i} ~~\forall i \iff F(k)\in H_s
$$
\end{proof}


Each element $k\in\Kdavis$ that is not a vertex lies in the interior of a unique simplex of $\Kdavis$. 
Having defined an injective  map $F:\Kdavis\to{\Klin}$ with the property that 
$$
\stab_W(F(k))=W_{\{s\in S\| F(k)\in H_s\}}=W_{\{s\in S\| k\in K_s\}}
$$
we define a map $F:\Sigma\to \Zlin$ by 
$$
F((w,k))=w\cdot F(k)
$$

\begin{thm}
The map $F:\Sigma \to \Zlin$ is a $W$-equivariant embedding. 
\end{thm}

\begin{proof}
It follows from the preceding lemma that this map is well-defined.
It is clearly $W$-equivariant. That it is injective follows from the observation that
if $k\in\Klin$ and $w\in W$ then $w\cdot k\in\Klin$ implies that $w\in\stab(k)$ (see Lemma~\ref{stab}).

\end{proof}




%
\bibliographystyle{amsplain}

\begin{thebibliography}{4}

\bibitem{NB}
N.~Bourbaki. \emph{Groupes et algebras de Lie, Chapitres 4, 5 et 6}. Hermann,
Paris, 1968.

\bibitem{Davis}
M.~Davis. \emph{The geometry and topology of Coxeter groups}. Vol. 32.
London Mathematical Society Monographs Series.
Princeton University Press, Princeton, NJ, 2008.









\bibitem{Dyer12}
M.~Dyer. ``Imaginary cone and reflection subgroups of Coxeter groups''.
arXiv: 1210.5206 [math.RT], preprint, 2012.


\bibitem{DHR13}
Matthew Dyer, Christophe Hohlweg and Vivien Ripoll. ``Imaginary cones and limit roots of infinite Coxeter groups''. 
In: \emph{Math. Z.} \textbf{284} (2016), no.3--4, pp. 715--780. 


\bibitem{Fu13}
X.~Fu. ``Coxeter groups, imaginary cones and dominance''. 
In: \emph{Pacific J. Math.} \textbf{262} (2013), no. 2, 339--363.



\bibitem{HLR11}
C.~Hohlweg, J.~P.~Labb{\'e} and V.~Ripoll. ``Asymptotical behaviour of roots of infinite Coxeter groups''. 
In: \emph{Canad. J. Math.} \textbf{66} Vol. 2, (2014), pp. 323--353. 


\bibitem{limit3}
C.~Hohlweg, J.~P.~Pr\'eaux and V.~Ripoll, ``On the Limit Set of Root Systems of Coxeter Groups and Kleinian Groups''. 
arXiv: 1305.0052 [math.GR], preprint, 2013.


\bibitem{RB96} 
R.~B.~Howlett, \emph{Introduction to Coxeter groups}. Lectures given at ANU,
1996 (available at http://www.maths.usyd.edu.au/res/Algebra/How/1997-6.html). 


\bibitem{K}
D.~Krammer.~``The conjugacy problem for Coxeter groups''.~In:~\emph{Groups Geom. Dyn.} 3.1 (2009), 
pp. 71--171. 


\end{thebibliography}

\providecommand{\bysame}{\leavevmode\hbox to3em{\hrulefill}\thinspace}
\providecommand{\MR}{\relax\ifhmode\unskip\space\fi MR }
\providecommand{\MRhref}[2]{%
  \href{http://www.ams.org/mathscinet-getitem?mr=#1}{#2}
}
\providecommand{\href}[2]{#2}

\end{document}